\documentclass[twoside,12pt]{article}
\usepackage{amsmath}
\usepackage{amsthm, amscd, amsfonts, amssymb, graphicx, color}% amsmath,amssymb,amsfonts,amsmath, amsthm, amscd, amsfonts, amssymb, graphicx, color
\usepackage[pagebackref=false,colorlinks,linkcolor=red,citecolor=green]{hyperref}
\usepackage{graphicx}
\def\bd{\begin{definition}}
\def\ed{\end{definition}}
\def\bth{\begin{theorem}}
\def\eth{\end{theorem}}
\def\bpf{\begin{proof}}
\def\epf{\end{proof}}
\def\bco{\begin{corollary}}
\def\eco{\end{corollary}}
\def\ble{\begin{lemma}}
\def\ele{\end{lemma}}
\def\bpr{\begin{proposition}}
\def\epr{\end{proposition}}
\newcommand{\be}{\begin{equation}}
\newcommand{\ee}{\end{equation}}
\newcommand{\bes}{\begin{equation*}}
\newcommand{\ees}{\end{equation*}}
\newcommand{\br}{\begin{remark}}
\newcommand{\er}{\end{remark}}
\newcommand{\ben}{\begin{enumerate}}
\newcommand{\een}{\end{enumerate}}
\newcommand{\beq}{\begin{eqnarray}}
\newcommand{\eeq}{\end{eqnarray}}
\newcommand{\beqn}{\begin{eqnarray*}}
\newcommand{\eeqn}{\end{eqnarray*}}

\def\pix{\frac{\partial}{\partial x^i}} %for abreviation commands
\def\nn{\nonumber}
\newtheorem{theorem}{Theorem}[section]
\newtheorem{lemma}[theorem]{Lemma}
\newtheorem{proposition}[theorem]{Proposition}
\newtheorem{corollary}[theorem]{Corollary}
\theoremstyle{definition}
\newtheorem{definition}[theorem]{Definition}

\theoremstyle{remark}
\newtheorem{remark}[theorem]{Remark}
\numberwithin{equation}{section}
\newcounter{tmp}

\newcounter{alphthm}
\begin{document}
\setcounter{page}{1}
\setlength{\unitlength}{10mm}
\title{Harmonic vector fields and the Hodge Laplacian operator on Finsler geometry}
%\title{The Hodge Laplacian operator on Finsler geometry}
%\title{A Hodge theorem on Finsler manifolds}
%\title{Harmonic $p$-forms on Finsler manifolds}
%\title{Harmonic vector fields on Finslerian spaces}
\author { \small Behroz Bidabad\thanks{The corresponding author, bidabad@aut.ac.ir, and  behroz.bidabad@math.univ-toulouse.fr} \, and Mir Ahmad Mirshafeazadeh}
\date{}
\maketitle
\begin{abstract}
We first present the natural definitions of the horizontal differential, the divergence (as an adjoint operator) and a $p$-harmonic form on a Finsler manifold.  Next, we prove a Hodge-type theorem for a Finsler manifold in the sense that a horizontal $p$-form is harmonic if and only if the horizontal Laplacian vanishes.
This viewpoint provides a new appropriate natural definition of harmonic vector fields in Finsler geometry.
  This approach leads to a Bochner-Yano type classification theorem based on the harmonic Ricci scalar.
 Finally, we show  that a closed orientable Finsler manifold with a positive harmonic Ricci scalar has zero Betti number.
\end{abstract}
\begin{center}
\textbf{Abrégé}
\end{center}
Nous présentons d'abord les définitions naturelles de la différentielle horizontale, de la divergence (comme opérateur adjoint) et d'une forme $p$-harmonique sur une variété  finslérienne. Ensuite, nous prouvons un théorème de type Hodge pour une variété finslérienne dans le sens où une $p$-forme horizontale est harmonique si et seulement si le Laplacien horizontal est nul.
Ce point de vue fournit une nouvelle définition naturelle appropriée des champs de vecteurs harmoniques en géométrie finslérienne.
  Cette méthode conduit à un théorème de classification de type Bochner-Yano basé sur le scalaire de Ricci harmonique.
 Enfin, nous montrons qu'une variété finslérienne  fermée et orientable avec un scalaire de Ricci harmonique positif, a un nombre de Betti nul.\\
 
\leftline{{\textbf{AMS Subject Classification 2020}: 58B20.
}}
\vspace{0.5cm}\textbf{Keywords}:
 Finsler;  harmonic p-form; Laplacian; Hodge's theorem; harmonic vector field; divergence; Bochner theorem; Betti number.

\section{Introduction}
The existence of harmonic vector fields on the Riemannian manifolds is directly related to the sign of the Ricci tensor. Bochner and Yano have studied the non-existence of harmonic vector fields on the compact Riemannian manifolds with positive Ricci curvature based on the Laplace-Beltrami operator. Next, Bochner proved that if the Ricci curvature on a Riemannian manifold is positive-definite, then all harmonic vector fields vanish \cite{B}. Yano proved that a vector field $X$ is harmonic, if and only if the Laplacian of its corresponding 1-form vanishes \cite{Y, Y1}.

 In Finsler geometry, Akbar-Zadeh introduced the divergence of horizontal and vertical 1-forms on $SM$  without defining the harmonic forms on a Finsler manifold, where $SM:=\bigcup\limits _{x\in M} S_xM$ and $S_xM:=\{y\in T_xM | F(y)=1\}$,  \cite{Ak}.

Harmonic forms in Finsler geometry are studied in \cite{BL,BR,CT,QW}. Recently, the second author introduced a definition of harmonic vector fields on a Finsler manifold, which is slightly modified here in the present work, see \cite{SB1,SB2}, and Remark \ref{Rem;1} in this article.
Moreover  some natural extensions of Riemannian results, more or less linked to this question are studied in \cite{BS}.

In the present work, the horizontal differential operator $d_{\texttt{H}}$  and the horizontal co-differential operator $\delta_{\texttt{H}}$, are defined as adjoint operators.
The above operators provide a Finslerian version of a well-known Hodge theorem on the Riemannian manifolds in the following sense.
\begin{theorem}\label{Th;p-form}
Let  $(M, F)$ be a closed Finsler manifold. If $ \omega $ is a horizontal p-form on $ SM $, then
\begin{align}
 \Delta_{\texttt{H}} \ \omega=0 \quad \textrm{if and only if} \quad  d_{\texttt{H}}\ \omega=0, \ and \quad \delta_{\texttt{H}}\ \omega=0 .
\end{align}
\end{theorem}
We can thus define harmonic $p$-forms naturally on a  Finsler manifold in the sense that, a horizontal p-form is harmonic if and only if the horizontal Laplacian vanishes.

The definition of harmonic p-forms on $SM$ will provide a new definition of a harmonic vector field on a Finsler manifold in the sense that, a vector field on $ (M, F) $  is harmonic if and only if the horizontal Laplacian vanishes.

Finally, we obtain a classification of harmonic vector fields based on the \emph{harmonic Ricci scalar} $\tilde{Ric}$ defined by the equation (\ref{de104}).
\begin{theorem}\label{theorem2}
Let $(M, F)$ be a closed  Finsler manifold and $X$  a harmonic vector field on $M.$
\begin{itemize}
\item[1.] If $ \tilde{Ric} = 0$, then $X$ is parallel.
\item[2.] If $ \tilde{Ric}>0 $, then $X$ vanishes.
\end{itemize}
\end{theorem}
This theorem is an extension of a well-known result obtained by Bochner and Yano, see  page \pageref{Th;A}, Theorem \ref{Th;A}.
%\begingroup
%\setcounter{tmp}{\value{theorem}}% store current value of theorem counter
%\setcounter{theorem}{0} %assign desired value to theorem counter
%\renewcommand\thetheorem{\Alph{theorem}}% locally redefine the representation of the theorem counter
%\begin{theorem}\label{Th;A}\cite{Y, Y1}
%Let $(M, g)$ be a closed  Riemannian manifold and $X$  a harmonic vector field on $M.$
%\begin{itemize}
%\item[1.] If $Ric(X, X)=X^{k}X^{t}R_{tk}=0 $, then $ X $ is parallel.
%\item[2.] If $ Ric(X, X)=X^{k}X^{t}R_{tk}>0$, then $ X $ vanishes.
%\end{itemize}
%\end{theorem}
%%\begin{theorem}\label{Pd}\cite{Pd}
%%If $ A $ is a closed, nonempty, convex subset of a Hilbert space $ B $, there is for every $ y $ in $ B $ a unique $ x $ in $ A $ that minimizes the distance from $ y $ to $ A. $
%%\end{theorem}
%\endgroup
 Finally, this brings us to the following fundamental results.
 \begingroup
\setcounter{tmp}{\value{theorem}}% store current value of theorem counter
\setcounter{theorem}{2} %assign desired value to theorem counter
\begin{theorem}\label{Th;harmonic}
Let $ (M, F) $ be a Finsler manifold. Every cohomology class $ H^{1}(M)$ contains a unique harmonic representative.
\end{theorem}
\begin{corollary}\label{Cor;BetiNum}
In a  closed orientable Finsler manifold with a positive harmonic Ricci scalar $ \tilde{Ric} >0$, the  first Betti number vanishes.
\end{corollary}
\endgroup
%\begin{remark}
%C. Bertrand and A. Rauzy, using a horizontal lift of a p-form on $M$ to $SM$ have defined the Laplacian on a Finsler manifold which is different from our point of view. More intuitively, they construct  a sub-elliptic operator on the associated unitary bundle and give  a lower bound for the first eigenvalue of this operator by using the horizontal Ricci tensor of the Berwald connection, see  \cite{BR}.
%\end{remark}
In Section 2, the necessary tools, concepts and definitions of Finsler geometry using the Cartan connection are stated. In Section 3, the definition of $ \Lambda^{H}_{p}(SM) $  the space of horizontal p-forms and  the definition of $ d_{H} $  the horizontal divergence operator
 on the unit fiber bundle $ SM $ with an inner product $ (., .) $ on $ \Lambda^{H}_{p}(SM) $ are expressed.
In section 4, the definition of the horizontal (co-differential) divergence, a horizontal Laplacian and a new type of harmonic p-form are introduced using the horizontal Laplacian. Section 5 deals with harmonic vector fields on Finsler manifolds where the proof of Theorem 1.2 is presented. In Section 6, we prove that a closed orientable Finsler manifold with a positive harmonic Ricci scalar has zero Betti number.
\section{Preliminaries and notations}
We first recall some Riemannian definitions of harmonic analysis.
Let $ (M, g) $ be a compact and orientable Riemannian manifold  of dimension $n$. A $p$-form on $ (M, g)$ for $1\leq p \leq n$  is given by
\begin{align*}
\varphi=\frac{1}{p!}\varphi_{i_{_{1}}...i_{p}}dx^{i_{_{1}}}\wedge ...\wedge dx^{i_{p}},
\end{align*}
where the indices  $i_1,...,i_p$ run over the range $1,...,n$ and the coefficients are components of  the skew-symmetric tensor fields of type $(0,p)$. The differential $ d\varphi $ is a
$(p+1)-$form   given by
\begin{align*}
 d\varphi=\frac{1}{(p+1)!}(\nabla_{i}\varphi_{i_{_{1}}...i_p}-\nabla_{i_{_{1}}}\varphi_{ii_{2}..i_{p}}-...
 -\nabla_{i_{p}}\varphi_{i_{_{1}}..i_{p-1} i})dx^{i}\wedge dx^{i_{1}}\wedge...\wedge dx^{i_{p}},
\end{align*}
where the coefficients are components of the skew-symmetric tensor fields of type $(0,p+1)$ and $\nabla_{j}$ are the components of Levi-Civita covariant derivative.
The co-differential $\delta\varphi $   is a $(p-1)-$form given by
\begin{align*}
\delta\varphi=-\frac{1}{(p-1)!}g^{ji}\nabla_{j}\varphi_{ii_{_{2}}..i_{p}} dx^{i_{2}}\wedge...\wedge dx^{i_{p}},
\end{align*}
where the coefficients are components of the skew-symmetric tensor fields of type $(0,p-1)$. The co-differential of a scalar function is defined to be zero.
It is easy to verify that $ d(d\varphi)=0 $ and $ \delta(\delta\varphi)=0 $, see for instance \cite{Y}.
In Riemannian geometry  a differential form $ \varphi $ is called \emph{harmonic} if it satisfies $ d\varphi=0 $ and $ \delta\varphi=0$.  A vector field $ X $ is said to be \emph{harmonic} if its associated 1-form is  harmonic. It is well known that a necessary and sufficient condition for a p-form $ \varphi $ to be harmonic is
\begin{align}
\Delta \varphi=(\delta d+d \delta)\varphi=0,
\end{align}
where $ \Delta $ is called Laplacian, see \cite{Y} for more details.

We then turn to the more general cases of Finsler manifolds. Let $M$ be a connected differentiable manifold,  $ \pi:TM_{0}\rightarrow M $ the bundle of non-zero tangent vector where $TM_0=TM \backslash 0$ is  the entire slit tangent bundle. A point of $TM$ is denoted
by $z=(x,y)$, where $x\in M$ and $y\in T_{x}M$. Let $(x^i)$ be a local chart with the domain $U\subseteq M$ and $(x^i,y^i)$ the induced local coordinates on $\pi^{-1}(U)$, where ${\bf y}=y^i\frac{\partial}{\partial x^i}\in T_{\pi z}M$, and $i$ running over the range $1,2,...,n$. A (globally defined) \emph{Finsler structure} on $M$ is a function $F:TM\longrightarrow [0,\infty)$ with the following properties;
$F$ is $C^\infty$ on the entire slit tangent bundle $TM \backslash 0 $;
$F(x,\lambda y )=\lambda F(x,y) ~ \forall \lambda >0$;
the $n\times n$ \emph{Hessian matrix} $(g_{ij}) =\frac {1}{2} ([F^2]_{y^i y^j})$ is positive-definite at every point of $TM_0$.
The pair $(M,g)$ is called a {\it Finsler manifold}, cf.  \cite{BCS}. Denote by $TTM_0$ and $SM$ the tangent bundle of $TM_0$ and the sphere bundle respectively, where $SM:=\bigcup\limits _{x\in M} S_xM$ and $S_xM:=\{y\in T_xM | F(y)=1\}$.

Let us consider the natural projection $ p: SM\rightarrow M$ which pulls back the tangent bundle $TM$ to an n-dimensional vector bundle $p^{*}TM$ over the $(2n-1)-$dimensional base $ SM $.
Given the natural induced coordinates $(x_{i}, y_{i})$ on $TM$, the coefficients of spray vector field are defined  by  (cf. \cite[p. 32]{SS})
\begin{align}
G^{i}:=\frac{1}{4}g^{ih}(\frac{\partial ^{2}F^{2}}{\partial y^{h}\partial x^{j}}y^{j}-\frac{\partial F^{2}}{\partial x^{h}}).
\end{align}
The pair $\lbrace \frac{\delta}{\delta x^{i}}, \frac{\partial}{\partial y^{i}}\rbrace$ forms a \emph{horizontal} and \emph{vertical} frame for $TTM$, where $\frac{\delta}{\delta x^{i}}:=\frac{\partial}{\partial x^{i}}-N^{j}_{i}\frac{\partial}{\partial y^{j}}$, and $N^{j}_{i}:=\frac{\partial G^{j}}{\partial y ^{i}}$ are called the coefficients of nonlinear connection. The tangent bundle $TTM_{0}$ of $TM_{0}$ can be split into the direct sum of the horizontal part $HTM$ spanned by $ \lbrace \frac{\delta}{\delta x^{i}}\rbrace $ and the vertical part $ VTM $ spanned by $ \lbrace \frac{\partial}{\partial y^{i}}\rbrace. $ The dual basis of $\lbrace \frac{\delta}{\delta x^{i}}, \frac{\partial}{\partial y^{i}}\rbrace  $ is $\lbrace dx^{i}, \delta y^{i}\rbrace $, where
\begin{align}
\delta y^{i}:=dy^{i}+N^{i}_{j}dx^{j},
\end{align}
and we have the following Whitney sum cf. \cite[p. 29]{SS}.
\begin{equation}
\begin{aligned}
TTM_{0}&=HTM\oplus VTM=span\lbrace \frac{\delta}{\delta x^{i}}\rbrace \oplus span \lbrace \frac{\partial}{\partial y^{i}}\rbrace ,\\
T^{*}TM_{0}&=H^{*}TM\oplus V^{*}TM=span\lbrace dx^{i}\rbrace \oplus span \lbrace \delta y^{i}\rbrace.
\end{aligned}
\end{equation}
%\subsection{Cartan connection}
The \emph{Cartan connection} is a natural extension of the Riemannian connection,  which is metric compatible and semi-torsion free. For a global approach to the Cartan connection one can refer to \cite{Ak}. According to the definition, the 1-forms of Cartan connection with respect to the dual basis $ \lbrace dx^{i}, \delta y^{i}\rbrace $
are given by
\begin{align*}
\omega^{i}_{j}:=\Gamma^{i}_{jk}dx^{k}+C^{i}_{jk}\delta y^{k},
\end{align*}
where, $\Gamma^{i}_{jk}$ and $C^{i}_{jk}$ are the \emph{horizontal} and \emph{vertical coefficients} of Cartan connection respectively defined by
\begin{equation*}
\begin{aligned}
& \Gamma^{i}_{jk}:=\frac{1}{2}g^{il}(\delta _{j}g_{lk}+\delta_{k}g_{jl}-\delta _{l}g_{jk}),\quad  C^{i}_{jk}:=\frac{1}{2}g^{il}\dot{\partial}_{l}g_{jk},
\end{aligned}
\end{equation*}
and
$ \delta _{i}:=\frac{\delta}{\delta x^{i}} $,
$ \dot{\partial}_{i}:=\frac{\partial}{\partial y^{i}}.$
In local coordinates  we have
\begin{equation*}
\begin{aligned}
& \nabla_{k}\dot{\partial_{j}}=\Gamma^{i}_{jk}\dot{\partial _{j}},
& \dot{\nabla}_{k}\dot{\partial _{j}}=C^{i}_{jk}\dot{\partial_{j}},\\
& \nabla_{k} \delta _{j} =\Gamma ^{i}_{jk}\delta_{i},
& \dot{\nabla}_{k} \delta _{j}=C^{i}_{jk}\delta _{i},
\end{aligned}
\end{equation*}
where in,
$ \nabla_{k}:=\nabla_{\frac{\delta}{\delta x^{k}} }$,
$ \dot{\nabla}_{k}:=\nabla_{\frac{\partial}{\partial y^{k}}}.$\\
Let us consider the components of an arbitrary (2,2)-tensor field $T_{is}^{jk}$ on $TM.$
The  \textit{horizontal} and \textit{vertical} components of the Cartan connection of $T_{is}^{jk}$ in a local coordinates  are given respectively by
\begin{equation*}
\begin{aligned}
& \nabla_{h}T_{is}^{jk}=\delta_{h}T_{is}^{jk}-T_{ps}^{jk}\Gamma ^{p}_{ih}-T_{ip}^{jk}\Gamma ^{p}_{sh}+T_{is}^{pk}\Gamma ^{j}_{ph}+T_{is}^{jp}\Gamma ^{k}_{ph},\\
& \dot{\nabla}_{h}T_{is}^{jk}=\dot{\partial }_{h}T_{is}^{jk}-T_{ps}^{jk}C ^{p}_{ih}-T_{ip}^{jk}C ^{p}_{sh}+T_{is}^{pk}C ^{j}_{ph}+T_{is}^{jp}C^{k}_{ph}.
\end{aligned}
\end{equation*}
The \emph{curvature tensor} in Cartan connection  is given by the  \textit{hh-curvature}, \textit{hv-curvature} and \textit{vv-curvature} with the following components, cf.  \cite{Ak};
\begin{equation*}
\begin{aligned}
R^{h}_{kij}&=\delta_{i}\Gamma ^{h}_{jk}-\delta_{j}\Gamma ^{h}_{ik}+\Gamma ^{l}_{jk}\Gamma ^{h}_{il}-\Gamma^{l}_{ik}\Gamma ^{h}_{jl}+R^{l}_{ij}C^{h}_{lk},\\
P^{h}_{kij}&=\dot{\partial_{k}}\Gamma ^{h}_{ki}-\delta_{i}C^{h}_{kj}+\Gamma ^{r}_{ki}C^{h}_{rj}-C^{r}_{kj}\Gamma ^{h}_{rj}+\dot{\partial_{j}}N^{r}_{i}C^{h}_{kr},\\
Q^{h}_{kij}&=C^{h}_{rj}C^{r}_{ki}-C^{h}_{ri}C^{r}_{kj},
\end{aligned}
\end{equation*}
respectively where
\begin{align}
R^{i}_{jk}=\frac{\delta N^{i}_{j}}{\delta x^{k}}-\frac{\delta N^{i}_{k}}{\delta x^{j}}=y^{m}R^{i}_{mjk}.
\end{align}
Trace of the $ hh-curvature $ of Cartan connection is denoted by $R_{ij}:=R^{l}_{ilj}$, which is not symmetric in general.\\
Let $ (M, F) $ be a Finsler manifold, $ \pi:TM_{0}\rightarrow M $ the bundle of non-zero tangent vectors and $ \pi^{*}TM $ the pullback bundle. The tangent space $T_xM$, $x\in M$ can be considered as a fiber of the pullback bundle $\pi^{*}TM$. Therefore a section $X$ on $\pi^{*}TM$ is denoted by $X=X^i(x,y)\pix$. The \emph{Ricci identity} for Cartan connection is given by the following equation
\begin{align}\label{de102}
\nabla_{k}\nabla_{h}X^{i}-\nabla_{h}\nabla_{k}X^{i}=X^{r}R^{i}_{rkh}-\dot{\nabla}_{r}X^{i}R^{r}_{kh},
\end{align}
 cf. \cite{Ak}.
Now we are in a position to define some basic notions on harmonic forms on Finsler manifolds.
\section{The p-forms and horizontal operators}
Here and everywhere in this paper, we assume the differential manifold $M$ is compact and without boundary or simply closed.
Let $(M, F)$ be a closed  Finsler manifold, $u: M\rightarrow SM$ a unitary vector field and $\omega =u_{i}dx^{i}$ the corresponding 1-form on $M$. A \emph{ volume element} on $SM$ is given by
%\begin{align*}
$\eta =\frac{(-1)^{\frac{n(n-1)}{2}}}{(n-1)!}\omega \wedge (d\omega)^{n-1},$ cf. \cite{Ak}.
%\end{align*}
We denote the space of all \emph{horizontal p-forms} on $SM$ by $\Lambda _{p}^{\texttt{H}}(SM)$ or simply $\Lambda _{p}^{\texttt{H}}$,
\begin{equation}
\Lambda ^{\texttt{H}}_{p}(SM):=\lbrace \varphi_{i_{_1}i_{_2}...i_{p}}(z)dx^{i_{_1}}\wedge dx^{i_{_2}}\wedge ... \wedge dx^{i_{p}}|\varphi_{i_{_1}i_{_2}...i_{p}}\in C^{\infty}(SM)\rbrace.
\end{equation}
Let $ \pi =a_{i}(z) dx^{i} $ be a horizontal 1-form on $ SM $. The \emph{co-differential} or \emph{divergence } of $ \pi $ concerning the Cartan connection is defined by
\begin{align}\label{de11}
 \delta \pi =-(\nabla ^{j}a_{j}-a_{j}\nabla_{0}T^{j}),
\end{align}
where, $T_{kij}=C_{kij}=\frac{1}{2}\frac{\partial g_{ij}}{\partial y^{k}},$
are the components of Cartan tensors and $\nabla_{0}=y^{i}\nabla_{i}$ cf.  \cite[p. 223]{Ak}. Also, we have
\begin{align}\label{de1}
\underset{SM}{\int} \delta \pi \ \eta =-\underset{SM}{\int}(\nabla ^{j}a_{j}-a_{j}\nabla_{0}T^{j}) \eta =-\underset{SM}{\int}(\nabla _{j}a^{j}-a^{j}\nabla_{0}T_{j}) \eta =0,
\end{align}
where $ a^{i}=g^{ij}a_{j} $, cf. \cite[p. 67]{Ak}. Let us denote the horizontal part of the differential $ d\pi $ by
\begin{align*}
\texttt{H}d\pi :=\frac{1}{2}(\nabla_{i}a_{j}-\nabla_{j}a_{i})(z)\ dx^{i}\wedge dx^{j},
\end{align*}
cf. \cite[p. 224]{Ak}. According to the above discussion, we are in a position to define a horizontal differential operator in  the following sense.
\bd
Let  $(M, F)$ be a Finsler manifold and $\varphi =\frac{1}{p!}\varphi_{i_{_1}...i_{p}}(z)dx^{i_{_1}}\wedge...\wedge dx^{i_{p}} \in \Lambda ^{H}_{p}$ a horizontal p-form on $SM$. A \emph{horizontal differential operator} is a differential operator on $SM$ given by
\begin{equation}
\begin{aligned}
d_{\texttt{H}}:& \Lambda _{p}^{\texttt{H}}\rightarrow \Lambda _{p+1}^{\texttt{H}}\\
& \varphi \rightarrow d_{\texttt{H}} \ \varphi,
\end{aligned}
\end{equation}
where, for $1\leq i, i_{k}\leq n$ and $1\leq k\leq p$,  we have
\begin{align}
 d_{\texttt{H}} \  \varphi=\frac{1}{(p+1)!}(\nabla_{i}\varphi_{i_{_{1}}...i_p}-\nabla_{i_{_{1}}}\varphi_{ii_{2}..i_{p}}-...
 -\nabla_{i_{p}}\varphi_{i_{_{1}}..i_{p-1}i})dx^{i}\wedge dx^{i_{1}}\wedge...\wedge dx^{i_{p}}.
\end{align}
\ed
Let $ \varphi $ and $ \pi $ be the  two arbitraries horizontal p-forms on
$SM$ with the components $\varphi_{i_{_{1}}...i_{p}}$ and $\pi_{i_{_{1}}...i_{p}}$,  respectively.  We consider an inner product $( . , . )$ on $ \Lambda _{p}^{\texttt{H}} $ as follows
\begin{align}\label{Def;InnerPro}
 (\varphi, \pi ):= \underset{SM}{\int} \ \frac{1}{p!} \ \varphi^{i_{_{1}}...i_{p}} \ \pi_{i_{_{1}}...i_{p}}\ \eta ,
\end{align}
where, $\varphi^{i_{_{1}}...i_{p}}=g^{i_{_{1}}j_{_{1}}}...g^{i_{p}j_{p}}\varphi_{j_{1}...j_{p}}$.
\section{The horizontal Laplacian and harmonic p-forms}
Using the above concepts, we   define the horizontal Laplacian. This definition of Laplacian is different from those given in \cite{Ak,BR} and \cite{SS}.

Let  $(M, F)$ be a Finsler manifold and $\psi$ a horizontal (p+1)-form on $SM$, given by
\begin{align*}
\psi =\frac{1}{(p+1)!}\psi_{ii_{_{1}}...i_{p}}dx^{i}\wedge dx^{i_{_{1}}}\wedge ...\wedge dx^{i_{p}}.
\end{align*}
 We define the \emph{horizontal divergence} (co-differential) of $\psi$ by
\begin{align}
(\delta_{\texttt{H}}\  \psi)_{j_{1}...j_{p}}:=-g^{ij}(\nabla _{i}\psi_{jj_{1}...j_{p}}-\psi_{jj_{1}...j_{p}}\nabla_{0}T_{i}).
\end{align}
\begin{remark}
 If $ \varphi $ is a horizontal 1-form on $ SM $, then $ \delta_{\texttt{H}}\  $  reduces to $ \delta $, and we have
\begin{align}\label{de101}
\delta_{\texttt{H}}\ \varphi=\delta\varphi=-(\nabla^{j}\varphi_{j}-\varphi_{j}\nabla_{0}T^{j}).
\end{align}
\end{remark}
\begin{definition}
Let  $(M, F)$ be a Finsler manifold. A \emph{horizontal Laplacian} on $SM$ is defined by
\begin{align}\label{de7}
\Delta_{\texttt{H}}:=d_{\texttt{H}}\delta_{\texttt{H}}+\delta_{\texttt{H}}d_{\texttt{H}},
\end{align}
where $d_{\texttt{H}}$ and $\delta_{\texttt{H}}$ are horizontal differential and horizontal co-differential operators on $SM$, respectively.
\end{definition}
  Now we are able to show the basic  equivalence relation
\begin{align}
 \Delta_{\texttt{H}} \ \omega=0 \quad \textrm{if and only if} \quad  d_{\texttt{H}}\ \omega=0, \ and \quad \delta_{\texttt{H}}\ \omega=0 ,
\end{align}
 in
 % $\Delta \varphi=0$ and $d\varphi =0, \delta\varphi=0$
  the following theorem.
%\begin{theorem}
%Let  $ (M, F) $ be a closed Finsler manifold. If $ \omega $ is a p-form on $ SM $, then
%\begin{align}
% \Delta_{\texttt{H}} \ \omega=0 \quad \textrm{if and only if} \quad  d_{\texttt{H}}\ \omega=0, \ and \quad \delta_{\texttt{H}}\ \omega=0 .
%\end{align}
%\end{theorem}

\emph{\textbf{Proof of Theorem \ref{Th;p-form}.}}~
%\begin{proof}
It is clear that if $ \delta_{\texttt{H}}=0 $ and  $ d_{\texttt{H}}\omega =0 $, then we have $ \Delta_{\texttt{H}}\ \omega=0 $. Conversely,  Let $\varphi =\frac{1}{p!}\varphi_{i_{_1}...i_{p}}(z)dx^{i_{_1}}\wedge...\wedge dx^{i_{p}} \in \Lambda ^{H}_{p}$ be a horizontal p-form on $SM$ and $\psi$ a horizontal (p+1)-form on $SM$, given by
\begin{align*}
\psi =\frac{1}{(p+1)!}\psi_{ii_{_{1}}...i_{p}}dx^{i}\wedge dx^{i_{_{1}}}\wedge ...\wedge dx^{i_{p}}.
\end{align*}
Antisymmetric property of p-forms yield
\begin{equation*}
\begin{aligned}
\nabla _{i_{k}}\varphi _{i_{1}...i_{k-1}ii_{k+1}...i_{p}}\psi ^{ii_{1}...i_{p}}&=\nabla _{i}\varphi _{i_{1}...i_{k-1}i_{k}i_{k+1}...i_{p}}\psi ^{i_{k}i_{1}...i_{k-1}ii_{k+1}...i_{p}}\\
&=(-1)^{k+(k-1)}\nabla _{i}\varphi _{i_{1}...i_{k-1}i_{k}i_{k+1}...i_{p}}\psi ^{ii_{1}...i_{k-1}i_{k}i_{k+1}...i_{p}}\\
&=-\nabla _{i}\varphi _{i_{1}...i_{k-1}i_{k}i_{k+1}...i_{p}}\psi ^{ii_{1}...i_{k-1}i_{k}i_{k+1}...i_{p}}.
\end{aligned}
\end{equation*}
Using the last equation and the inner product \eqref{Def;InnerPro} we have
\begin{equation}
\begin{aligned}\label{de2}
(d_{\texttt{H}}\ \varphi , \psi )&=\underset{SM}{\int}\frac{1}{(p+1)!}(\nabla_{i}\varphi_{i_{_{1}}...i_{p}}-...-\nabla_{i_{p}}\varphi_{i_{_{1}}...i_{p-1}i})\ \psi ^{ii_{_{1}}...i_{p}} \ \eta  \\
&=\underset{SM}{\int}\frac{1}{(p+1)!}(\nabla_{i}\varphi_{i_{_{1}}...i_{p}}+...+\nabla_{i}\varphi_{i_{_{1}}...i_{p}})\ \psi ^{ii_{_{1}}...i_{p}} \ \eta  \\
& =\underset{SM}{\int}\frac{1}{p!}\nabla_{i}\varphi_{i_{_{1}}...i_{p}} \ \psi^{ii_{_{1}}...i_{p}} \ \eta.
\end{aligned}
\end{equation}
Letting $ a^{i}=\varphi_{i_{1}...i_{p}}\psi ^{ii_{1}...i_{p}}$, equation (\ref{de1}) yields
\begin{align}\label{da1}
\underset{SM}{\int} \nabla_{i}(\varphi _{i_{_{1}}...i_{p}}\psi^{ii_{_{1}}...i_{p}})\eta =\underset{SM}{\int}\varphi _{i_{_{1}}...i_{p}}\psi ^{ii_{_{1}}...i_{p}}\nabla_{0}T_{i}\eta.
\end{align}
Replacing (\ref{da1}) in (\ref{de2}) and using the metric compatibility of Cartan connection yields
\begin{equation}\label{de3}
\begin{aligned}
p! (d_{\texttt{H}} \ \varphi, \psi )=&\underset{SM}{\int}\nabla_{i}(\varphi_{i_{_{1}}...i_{p}}\psi^{ii_{_{1}}...i_{p}})\eta -\underset{SM}{\int}\varphi_{i_{_{1}}...i_{p}}\nabla_{i}\psi^{ii_{_{1}}...i_{p}}\eta \\
= &\underset{SM}{\int}\varphi_{i_{_{1}}...i_{p}}\psi ^{ii_{_{1}}...i_{p}}\nabla_{0}T_{i}\eta -\underset{SM}{\int}\varphi_{i_{_{1}}...i_{p}}\nabla_{i}\psi^{ii_{_{1}}...i_{p}}\eta  \\
=& -\underset{SM}{\int} (\nabla_{i}\psi^{ii_{_{1}}...i_{p}}-\psi^{ii_{_{1}}...i_{p}}\nabla_{0}T_{i})\varphi_{i_{_{1}}...i_{p}}\eta \\
=& -\underset{SM}{\int}g^{ij}g^{i_{_{1}}j_{1}}...g^{i_{p}j_{p}}(\nabla_{i}\psi_{jj_{1}...j_{p}}-\psi_{jj_{1}...j_{p}}\nabla_{0}T_{i})\varphi_{i_{_{1}}...i_{p}}\eta.
\end{aligned}
\end{equation}
Therefore (\ref{de2}) becomes
\begin{equation*}
\begin{aligned}
p!(d_{\texttt{H}} \ \varphi, \psi)=& \underset{SM}{\int}g^{i_{_{1}}j_{1}}...g^{i_{p}j_{p}}
(\delta_{\texttt{H}}\psi)_{j_{1}...j_{p}} \ \varphi _{i_{_{1}}...i_{p}} \ \eta \\
=& p! (\delta_{\texttt{H}}\ \psi , \varphi ),
\end{aligned}
\end{equation*}
which yields
\begin{align}\label{de4}
(d_{\texttt{H}}\ \varphi , \psi )=(\varphi , \delta_{\texttt{H}}\ \psi).
\end{align}
If $ \varphi =\omega $ is a p-form and $ \psi =d_{H}\omega $, then the equation (\ref{de4}) yields
\begin{align}\label{de5}
(d_{\texttt{H}}\ \omega , d_{\texttt{H}}\ \omega)=(\omega , \delta_{\texttt{H}} d_{\texttt{H}}\ \omega).
\end{align}
If $ \varphi =\delta_{\texttt{H}}\ \omega $ and $ \psi=\omega $, using (\ref{de4}) we have
\begin{align}\label{de6}
(d_{\texttt{H}} \delta_{\texttt{H}}\ \omega , \omega)=(\delta_{\texttt{H}}\ \omega , \delta_{\texttt{H}}\ \omega).
\end{align}
 Through the equations  (\ref{de5}), (\ref{de6}) and  (\ref{de7}) we have
\begin{equation*}
\begin{aligned}
(\Delta_{\texttt{H}} \omega , \omega)=& (d_{\texttt{H}} \delta_{\texttt{H}}\ \omega ,\omega )+(\delta_{\texttt{H}} d_{\texttt{H}}\omega , \omega)\\
=& (\delta_{\texttt{H}}\ \omega,\delta_{\texttt{H}}\ \omega)+(d_{\texttt{H}}\ \omega ,d_{\texttt{H}}\  \omega)\geq 0.
\end{aligned}
\end{equation*}
If $\Delta_{\texttt{H}}\ \omega=0  $, we conclude that $ \delta_{\texttt{H}}\ \omega=0 $ and $ d_{\texttt{H}}\ \omega =0$
 which completes the proof.\hspace{\stretch{1}}$\Box$\\
%\end{proof}
\subsection{Horizontal Laplacian of p-forms}
Let $ \varphi $ be a horizontal p-form on $ SM $, by definitions of horizontal differential and co-differential we can easily see that
\begin{equation}\label{de8}
\begin{aligned}
\delta_{\texttt{H}}\ d_{\texttt{H}}\ \varphi =-\frac{1}{p!}& [(g^{rs}(\nabla_{r}\nabla_{s}\varphi_{i_{_{1}}...i_{p}}-\nabla_{s}\varphi_{i_{_{1}}...i_{p}}\nabla_{0}T_{r})\\
& -g^{rs}(\nabla_{r}\nabla_{i_{_{1}}}\varphi_{si_{2}...i_{p}}-\nabla_{i_{_{1}}}\varphi_{si_{2}...i_{p}}\nabla_{0}T_{r})\\
& -g^{rs}(\nabla_{r}\nabla_{i_{2}}\varphi_{i_{_{1}}si_{3}...i_{p}}-\nabla_{i_{2}}\varphi_{i_{_{1}}si_{3}...i_{p}}\nabla_{0}T_{r})-...\\
& -g^{rs}(\nabla_{r}\nabla_{i_{p}}\varphi_{i_{_{1}}...i_{p-1}s}-\nabla_{i_{p}}\varphi_{i_{_{1}}...i_{p-1}s}\nabla_{0}T_{r})]dx^{i_{_{1}}}\wedge ...\wedge dx^{i_{p}},
\end{aligned}
\end{equation}
and
\begin{align*}
\delta_{\texttt{H}}\ \varphi=-\frac{1}{(p-1)!}g^{rs}(\nabla_{r}\varphi_{si_{2}...i_{p}}-\varphi_{si_{2}...i_{p}}\nabla_{0}T_{r})dx^{i_{2}}\wedge ...\wedge dx^{i_{p}}.
\end{align*}
On the other hand, by definition we have
\begin{equation}\label{de9}
\begin{aligned}
d_{\texttt{H}}\ \delta_{\texttt{H}}\ \varphi =&-\frac{1}{p!}[g^{rs}(\nabla_{i_{_{1}}}\nabla_{r}\varphi_{si_{2}...i_{p}}-\nabla_{i_{_{1}}}(\varphi_{si_{2}...i_{p}}\nabla_{0}T_{r}))\\
& -g^{rs}(\nabla_{i_{2}}\nabla_{r}\varphi_{si_{_{1}}i_{3}...i_{p}}-\nabla_{i_{2}}(\varphi_{si_{_{1}}i_{3}...i_{p}}\nabla_{0}T_{r}))-...\\
& -g^{rs}(\nabla_{i_{p}}\nabla_{r}\varphi_{si_{2}...i_{p-1}i_{_{1}}}-\nabla_{i_{p}}(\varphi_{si_{2}...i_{p-1}i_{_{1}}}\nabla_{0}T_{r}))]dx^{i_{_{1}}}\wedge ...\wedge dx^{i_{p}}.
\end{aligned}
\end{equation}
The equations (\ref{de8}) and (\ref{de9}) yield
\begin{equation}
\begin{aligned}
(\delta_{\texttt{H}}\ d_{\texttt{H}}\ +d_{\texttt{H}}\ \delta_{\texttt{H}}\ )\varphi =&-\frac{1}{p!}[g^{rs}(\nabla_{r}\nabla_{s}\varphi_{i_{_{1}}...i_{p}}-\nabla_{s}\varphi_{i_{_{1}}...i_{p}}\nabla_{0}T_{r})\\
&-g^{rs}(\nabla_{r}\nabla_{i_{_{1}}}\varphi_{si_{2}...i_{p}}-\nabla_{i_{_{1}}}\nabla_{r}\varphi_{si_{2}...i_{p}})\\
&-g^{rs}(\nabla_{r}\nabla_{i_{2}}\varphi_{i_{_{1}}si_{3}...i_{p}}-\nabla_{i_{2}}\nabla_{r}\varphi_{i_{_{1}}si_{3}...i_{p}})-...\\
&-g^{rs}(\nabla_{r}\nabla_{i_{p}}\varphi_{i_{_{1}}...i_{p-1}s}-\nabla_{i_{p}}\nabla_{r}\varphi_{i_{_{1}}...i_{p-1}s})\\
&-g^{rs}(\varphi _{si_{2}...i_{p}}\nabla_{i_{_{1}}}\nabla_{0}T_{r}+\varphi _{i_{_{1}}si_{3}...i_{p}}\nabla_{i_{2}}\nabla_{0}T_{r}\\
& +...+\varphi _{i_{_{1}}...i_{p-1}s}\nabla_{i_{p}}\nabla_{0}T_{r})]dx^{i_{_{1}}}\wedge ...\wedge dx^{i_{p}}.
\end{aligned}
\end{equation}
In particular for an arbitrary horizontal 1-form $\varphi=\varphi_{i}(z)dx^{i}$  on $SM$, the above equation reduces to
\begin{equation}\label{de10}
\begin{aligned}
(\delta_{\texttt{H}}\ d_{\texttt{H}}  +d_{\texttt{H}}\ \delta_{\texttt{H}} )\varphi =&-[g^{rs}(\nabla_{r}\nabla_{s}\varphi_{i}-\nabla_{s}\varphi_{i}\nabla_{0}T_{r})\\
&-g^{rs}(\nabla_{r}\nabla_{i}\varphi_{s}-\nabla_{i}\nabla_{r}\varphi_{s})\\
&-g^{rs}(\varphi _{s}\nabla_{i}\nabla_{0}T_{r})]dx^{i}.
\end{aligned}
\end{equation}
This fact gives rise to a new definition of horizontal harmonic vector fields on Finsler manifolds.
\begin{definition}
A horizontal p-form $ \varphi $ on $ SM $ is called \emph{horizontally harmonic} if we have
\begin{align*}
\Delta_{\texttt{H}}\  \varphi =0.
\end{align*}
The horizontal harmonic p-forms will be referred to in the suite as h-harmonic p-forms or simply h-harmonic.
\end{definition}
\br
C. Bertrand and A. Rauzy, using a horizontal lift of a p-form on $M$ to $SM$ have defined the Laplacian on a Finsler manifold which is different from our point of view. More intuitively, they construct  a sub-elliptic operator on the associated unitary bundle and give  a lower bound for the first eigenvalue of this operator by using the horizontal Ricci tensor of the Berwald connection, see  \cite{BR}.
\er
\section{The harmonic vector fields on Finsler manifolds}
Recently, one of the present authors has introduced in a joint work a definition for harmonic vector fields on Finsler manifolds using the Cartan and Berwald connections in the following sense.
\br\label{Rem;1}
Let $(M, F)$ be a closed Finsler manifold. A vector field $ X=X^{i}\frac{\partial}{\partial X^{i}} $ on $M$  is called harmonic   if its corresponding horizontal 1-form $ X=X_{i}(z)dx^{i}$ on $SM$ satisfies $\Delta X=0$ or  $dX=0$ and $\delta X=0$, where
\begin{equation}
\begin{aligned}
dX&=\frac{1}{2}(D_{i}X_{j}-D_{j}X_{i})dx^{i}\wedge dx^{j}-\frac{\partial X_{i}}{\partial y^{j}}dx^{i}\wedge dy^{j},\\
\delta X&=-(\nabla ^{j}X_{j}-X_{j}\nabla _{0}T^{j})=-g^{ij}D_{i}X_{j},
\end{aligned}
\end{equation}
and $ \nabla $ and $ D $ are the covariant derivatives of Cartan and Berwald connections, respectively, cf.  \cite{SB1,SB2}.
\er
 The above definition of harmonic vector fields and the corresponding harmonic 1-forms have some inconveniences. First, it could not be easily extended to the harmonic p-forms on Finsler manifolds.  In particular, the occurrence of the mixed terms of differential and co-differential could not be readily established in the Finsler setting.  Second, the both Berwald's and Cartan's covariant derivatives must be considered in this calculations which needs more preliminaries for this definition. Finally,  contrary to the definition of harmonic vector fields on the Riemannian manifolds, we do not have the following proper bilateral relation in general;
\begin{align}
\Delta \varphi=d\delta \varphi +\delta d\varphi =0\quad \Longleftrightarrow \qquad    d\varphi =0\quad  \textrm{and} \quad \delta \varphi =0.
\end{align}
%the defined co-differential  $ \delta _{H}=\delta $  is not adjoint to $ d $, that is to say $(d \varphi , \psi )\neq (\varphi , \delta \psi ) $ and we do not have the following useful equivalence relation;
%\begin{align}
%\Delta \varphi=d\delta \varphi +\delta d\varphi =0\quad \Longleftrightarrow \quad  \lbrace d\varphi =0, \delta \varphi =0\rbrace.
%\end{align}
% On the other hand Lemma \ref{lemma1} and (\ref{de101}) yield that if $ \varphi $ is an arbitrary horizontally 1-form on $ SM $, then $ \delta _{H}=\delta $  is not adjoint by $ d $, that is to say
%$ (d \varphi , \psi )\neq (\varphi , \delta \psi ) $ and we do not have the following equivalence
%\begin{align}
%\Delta \varphi=d\delta \varphi +\delta d\varphi =0 \Longleftrightarrow \lbrace d\varphi =0, \delta \varphi =0\rbrace,
%\end{align}
%as we dose in Riemannian manifolds.
The remedy lies in a slight modification of definition in the following sense.
%slightly change is slightly change the to define the harmonic p-forms as follows.
Let $X=X^i (x)\frac{\partial}{\partial x^i}$ be a vector field on $M$. One  can associate to $ X $ a 1-form  $\tilde{X}$ on $ SM$ defined by
\begin{align*}
\tilde{X} =X_i (z)dx^i + \dot{X}_i \frac{\delta y^{i}}{F},
\end{align*}
where
%\begin{align*}
$ \dot{X}_i =\frac{1}{F}({\nabla_0 X_i -y_i \nabla _0 (y^j X_j)F^{-2}})$,
%\end{align*}
 and $z\in SM $ \cite{Ak}.
The horizontal part of the associated 1-form $\tilde{X}$ on $ SM $ is called \emph{associate horizontal 1-form} and denoted by $X=X_i (z)dx^i$.
\begin{definition}
Let $(M, F)$ be a Finsler manifold. A vector field $ X=X^{i}(x)\frac{\partial}{\partial x^{i}}$ on $M$ is called \emph{harmonic} related to the Finsler structure $F$ if the associate horizontal 1-form  $X=X_i (z)dx^i$  is $h$-harmonic on $SM$.
\end{definition}
\begin{remark}\label{R1}
According to this definition of the Finslerian harmonic vector field,  if $ X $ is a harmonic vector field concerning the Finsler structure $F$, then the associate horizontal 1-form  $X=X_i (z)dx^i$, is h-harmonic on $SM$, where     $X_i (z)$ is a real function on $SM$ and $z=(x,y)\in SM $.
\end{remark}
\begin{theorem}
Let $(M, F)$ be a closed Finsler manifold. A vector field $\varphi=\varphi^{i}\frac{\partial}{\partial x^{i}}$  on $M$ is harmonic if and only if
\begin{align}\label{de103}
g^{rs}(\nabla_{r}\nabla_{s}\varphi_{i}-\nabla_{s}\varphi_{i}\nabla_{0}T_{r})=\varphi ^{t}R_{ti}-\dot{\nabla}_{t}\varphi^{r}R^{t}_{ri}+\varphi^{r}\nabla_{i}\nabla_{0}T_{r}.
\end{align}
%where $ R_{ti} =R^{r}_{tri} .$
\end{theorem}
\begin{proof}
The Ricci identity (\ref{de102}) yields
\begin{equation}
\begin{aligned}
g^{rs}(\nabla_{r}\nabla_{i}\varphi_{s}-\nabla_{i}\nabla_{r}\varphi_{s})&= \nabla_{r}\nabla_{i}\varphi^{r}-\nabla_{i}\nabla_{r}\varphi^{r}\\
&=\varphi^{t}R^{r}_{tri}-\dot{\nabla}_{t}\varphi^{r}R^{t}_{ri}\\
&=\varphi^{t}R_{ti}-\dot{\nabla}_{t}\varphi^{r}R^{t}_{ri}.
\end{aligned}
\end{equation}
Substituting the last equation in (\ref{de10}) we get the result.
\end{proof}
A Finsler manifold $(M, F)$  is called a \emph{Landsberg manifold} if $\nabla_{0}T=0$ . We have the following corollary.
\begin{corollary}
Let $(M, F)$ be a closed Landsberg manifold. A vector field $\varphi=\varphi^{i}\frac{\partial}{\partial x^{i}}$
on $M$ is harmonic if and only if
\begin{align*}
g^{rs}\nabla_{r}\nabla_{s}\varphi_{i}=\varphi ^{t}R_{ti}-\dot{\nabla}_{t}\varphi^{r}R^{t}_{ri}.
\end{align*}
If $ (M, F) $ is Riemannian, then the above equation reduces to the following well known form.%, see for instance \cite[P. 43]{Y}.
\begin{align*}
g^{rs}\nabla_{r}\nabla_{s}\varphi_{i}=\varphi ^{t}R_{ti}.
\end{align*}
 \end{corollary}
Let $ X=X^{i}(x) \frac{\partial}{\partial x^{i}}$ be a vector field on $ (M, F)$.
Inspired by \cite{SB1} and \cite{SB2} and based on the Ricci tensor, we define the \textit{  harmonic Ricci scalar} $\tilde{Ric}$ as follows
\begin{align}\label{de104}
 \tilde{Ric}(X,X):=X^{k}X^{t}R_{tk}-X^{k}\dot{\nabla}_{r}X^{j}R^{r}_{jk}-X^{k}\nabla_{k}X^{j}\nabla_{0}T_{j}.
\end{align}
Furthermore, we obtain a  classification result given in Theorem \ref{theorem2}.\\

\emph{\textbf{Proof of Theorem \ref{theorem2}.}}~
Let $ X=X^{i}(x) \frac{\partial}{\partial x^{i}}$ be a vector field on $ (M, F)$  and $ Y $ and $ Z $ two 1-forms  on $ SM $ defined at $ z\in SM $ by $Y=(X^{k}\nabla_{k}X_{i})(z) dx^{i}$ and $ Z=(X_{i}\nabla_{j}X^{j})(z)dx^{i}$, respectively. Using (\ref{de11}) we have
\begin{equation}
\begin{aligned}
\delta Y=& -\nabla_{j}(X^{k}\nabla_{k}X^{j})+X^{k}\nabla_{k}X^{j}\nabla_{0}T_{j}\\
=& -\nabla_{j}X^{k}\nabla_{k}X^{j}-X^{k}\nabla_{j}\nabla_{k}X^{j}+X^{k}\nabla_{k}X^{j}\nabla_{0}T_{j},
\end{aligned}
\end{equation}
and similarly
\begin{equation}
\begin{aligned}
\delta Z=&-\nabla_{k}X^{k}\nabla_{j}X^{j}-X^{k}\nabla_{k}\nabla_{j}X^{j}+X^{k}\nabla_{j}X^{j}\nabla_{0}T_{k}\\
=& -\nabla_{k}X^{k}(\nabla_{j}X^{j}-X^{k}\nabla_{0}T_{k})-X^{k}\nabla_{k}\nabla_{j}X^{j}\\
=& \nabla_{k}X^{k}\delta X-X^{k}\nabla_{k}\nabla_{j}X^{j}.
\end{aligned}
\end{equation}
The difference of $ \delta Z $ and $ \delta Y$ yields
\begin{equation}\label{de12}
\begin{aligned}
\delta Z-\delta Y=& \nabla_{k}X^{k}\delta X+X^{k}(\nabla_{j}\nabla_{k}X^{j}-\nabla_{k}\nabla_{j}X^{j})\\
& +\nabla_{j}X^{k}\nabla_{k}X^{j}-X^{k}\nabla_{k}X^{j}\nabla_{0}T_{j}.
\end{aligned}
\end{equation}
On the other hand we have
\begin{align*}
d_{\texttt{H}}\ X=\frac{1}{2}(\nabla_{i}X_{j}-\nabla_{j}X_{i})dx^{i}\wedge dx^{j},
\end{align*}
from which
\begin{equation}
\begin{aligned}
||d_{\texttt{H}}\ X||^{2}&=\frac{1}{4}(\nabla_{i}X_{j}-\nabla_{j}X_{i})(\nabla ^{i}X^{j}-\nabla ^{j}X^{i})\\
&=\frac{1}{4}[(\nabla_{i}X_{j})(\nabla ^{i}X^{j})-(\nabla_{i}X_{j})(\nabla ^{j}X^{i})-(\nabla_{j}X_{i})(\nabla ^{i}X^{j})+(\nabla_{j}X_{i})(\nabla ^{j}X^{i})] \nn\\
&=\frac{1}{2}[||\nabla X||^{2}-(\nabla_{i}X_{j})(\nabla ^{j}X^{i})].\nn
\end{aligned}
\end{equation}
Therefore
\begin{align}\label{de13}
\nabla_{j}X^{k}\nabla_{k}X^{j}=||\nabla X||^{2}-2||d_{\texttt{H}}\ X||^2.
\end{align}
Replacing (\ref{de13}) and (\ref{de102}) in (\ref{de12}) we obtain
\begin{equation}\label{de20}
\begin{aligned}
\delta Z-\delta Y=& \nabla_{k}X^{k}\delta X+X^{k}X^{t}R_{tk}-X^{k}\dot{\nabla}_{r}X^{j}R^{r}_{jk}\\
& +||\nabla X||^{2}-2||d_{\texttt{H}}\ X||^2-X^{k}\nabla_{k}X^{j}\nabla_{0}T_{j}.
\end{aligned}
\end{equation}
If $ X $ is a harmonic vector field, then by definition of $\tilde{Ric}$ given by (\ref{de104}) the last equation becomes
\begin{align*}
\delta Z-\delta Y=||\nabla X||^{2}+\tilde{Ric}.
\end{align*}
By integration over $ SM $ and using (\ref{de1}), we obtain
\begin{align}\label{de20}
\underset{SM}\int(\tilde{Ric}+||\nabla X||^{2})\eta =0.
\end{align}
 If $ \tilde{Ric}=0, $ or
\begin{align*}
 X^{k}X^{t}R_{tk}= X^{k}\dot{\nabla}_{r}X^{j}R^{r}_{jk}+X^{k}\nabla_{k}X^{j}\nabla_{0}T_{j},
 \end{align*}
 then (\ref{de20}) yields the first assertion.
 If $ \tilde{Ric}>0, $ that is, if we have
  \begin{align*}
 X^{k}X^{t}R_{tk}>X^{k}\dot{\nabla}_{r}X^{j}R^{r}_{jk}+X^{k}\nabla_{k}X^{j}\nabla_{0}T_{j}, \end{align*}
then using the equation (\ref{de20}) we get the second assertion.
\hspace {\stretch{1}}$\Box$\\
\begin{remark}
For a closed  Landsberg manifold and a harmonic vector field  $X$  on $M$, Theorem \ref{theorem2} reads
\begin{itemize}
\item[1.] If $X^{k}X^{t}R_{tk}= X^{k}\dot{\nabla}_{r}X^{j}R^{r}_{jk} $, then $ X $ is parallel.
\item[2.] If $ X^{k}X^{t}R_{tk}>X^{k}\dot{\nabla}_{r}X^{j}R^{r}_{jk}$, then $ X $ vanishes.
\end{itemize}
\end{remark}
Recall that if the Finsler structure $F$ is Riemannian, then Theorem \ref{theorem2}  reduces to the following  famous theorem of Bochner and Yano.
\begingroup
\setcounter{tmp}{\value{theorem}}% store current value of theorem counter
\setcounter{theorem}{0} %assign desired value to theorem counter
\renewcommand\thetheorem{\Alph{theorem}}% locally redefine the representation of the theorem counter
\begin{theorem}\label{Th;A}\cite{Y, Y1}
Let $(M, g)$ be a closed  Riemannian manifold and $X$  a harmonic vector field on $M.$
\begin{itemize}
\item[1.] If $Ric(X, X)=X^{k}X^{t}R_{tk}=0 $, then $ X $ is parallel.
\item[2.] If $ Ric(X, X)=X^{k}X^{t}R_{tk}>0$, then $ X $ vanishes.
\end{itemize}
\end{theorem}
%\begin{theorem}\label{Pd}\cite{Pd}
%If $ A $ is a closed, nonempty, convex subset of a Hilbert space $ B $, there is for every $ y $ in $ B $ a unique $ x $ in $ A $ that minimizes the distance from $ y $ to $ A. $
%\end{theorem}
\endgroup
\section{Cohomology class and Betti number}
On a smooth manifold $M$ the de Rham cohomology $ H^{1}_{dR}(M):= Z^{1}(M) / B^{1}(M)$, is an equivalence class  of the closed forms on $M$.
The fact that a closed form is not exact indicates that the manifold has a certain global topological structure that prevents the existence of any hole or twist. The de Rham cohomology class is therefore, a way to understand, via the tangent bundle, the global topology of a manifold.

On a compact Riemannian manifold,  every equivalence class in $H^k_{\mathrm{dR}}(M)$ contains exactly one harmonic form. That is, every member $  \omega $ of a given equivalence class of closed forms can be written as
$ {\displaystyle \omega =\alpha +\gamma }$
where ${\displaystyle \alpha }  $ is exact and $ \gamma $ is harmonic, i.e. $ {\displaystyle \Delta \gamma =0}$.

The dimension of the space of all harmonic forms of degree $p$ on a manifold $M$ is called the \emph{pth Betti number} of the manifold.

Due to Hodge theory, the first Betti number is equal to the dimension of the space of harmonic $1$-forms on $M$, and this space is isomorphic to $H^1_{dR}(M)$.

 As mentioned earlier, on a Finsler manifold $(M, F)$, a vector field is harmonic if  $X=X_i (x,y)dx^i$, the associate horizontal $1$-form  on $SM$,  is $h$-harmonic. Hence the definition of a harmonic form  on $(M, F)$  is closely related to the Finsler structure $F$.

%  If a smooth manifold $M$ is equipped with a Finsler structure $F$ and an associated Finsler connection, then we have the horizontal and vertical decomposition subspaces on the sphere bundle $SM$.
%
% As well,  the first de Rahm cohomology class on $(M, F)$  is defined by an equivalence class  of the closed horizontal $1$-forms on $SM$,  that is
% $ H^{1}_{dR}(M):=  Z_H^{1}(SM) / B_H^{1}(SM) $, where $Z_H^{1}(SM) $  is the set of closed horizontal $1$-forms and $B_H^{1}(SM) $  is the set of exact horizontal $1$-forms  on the horizontal part of  $SM$, respectively.

 %We call the dimension of the space of all horizontal $1$-forms on the sphere bundle $SM$, the \emph{h-Betti number}.

 % One can easily see that if $F$  is Riemannian then  the $h$-Betti number  reduces to the Betti number of $M$.   (PLEASE EXPLAIN WHY)

The following theorem will be used in the sequel.
\begingroup
\setcounter{tmp}{\value{theorem}}% store current value of theorem counter
\setcounter{theorem}{1} %assign desired value to theorem counter
\renewcommand\thetheorem{\Alph{theorem}}% locally redefine the representation of the theorem counter
\begin{theorem}\label{Pd}\cite{Pd}
If $ A $ is a closed, nonempty, convex subset of a Hilbert space $ B $, then  for every $ y $ in $ B $ there is a unique $ x $ in $ A $ that minimizes the distance from $ y $ to $ A. $
\end{theorem}
\endgroup

%\begin{theorem}\label{Th;harmonic}
%Let $ (M, F) $ be a Finsler manifold. Every cohomology class $ H^{1}(M)$ contains a unique harmonic representative.
%\end{theorem}

We are now able to prove  Theorem \ref{Th;harmonic}.

\emph{\textbf{Proof of Theorem \ref{Th;harmonic}.}}~
\emph{Uniqueness.}
Let $ (M, F) $ be a Finsler manifold,  $ \alpha^{(1)}=\alpha^{(1)}_{i}(x)dx^{i} $  and $ \alpha^{(2)}=\alpha^{(2)}_{i}(x)dx^{i} $  the two 1-forms on $M$ such that the associate horizontal 1-forms $ \alpha^{(1)}=\alpha^{(1)}_{i}(x, y)dx^{i} $ and
$\alpha^{(2)}=\alpha^{(2)}_{i}(x, y)dx^{i}$ on $SM$  are $h$-harmonic and $\alpha^{(1)}_{i}(x, y)dx^{i}-\alpha^{(2)}_{i}(x, y)dx^{i}=d_{H}f $  for some
 $f\in C^\infty(SM)$. Using the inner product   \eqref{Def;InnerPro} and the equation \eqref{de4}, we have
 \begin{align}\label{eqz1}
(\alpha^{(1)}_{i}(x, y)dx^{i}-\alpha^{(2)}_{i}(x, y)dx^{i},& \alpha^{(1)}_{i}(x, y)dx^{i}-\alpha^{(2)}_{i}(x, y)dx^{i})\\
&=(\alpha^{(1)}_{i}(x, y)dx^{i}-\alpha^{(2)}_{i}(x, y)dx^{i}, d_{H}f)\nn\\
&=(\delta_{H}(\alpha^{(1)}_{i}(x, y)dx^{i}-\alpha^{(2)}_{i}(x, y)dx^{i}), f)\nn\\
&=(0, f)=0, \nn
\end{align}
which yields  $ \alpha^{(1)}=\alpha^{(2)}. $\\
\emph{Existence.}
$ B^{1}(SM) $ is closed in $ Z^{1}(SM) $ and it is convex \cite{Pd}.\\
%IS IT CORRECT FOR $Z_H^{1}(SM) $?\\
 Let $ \theta=\theta_{i}(x)dx^{i} \in Z^{1}(M) $  such that $ \theta=\theta_{i}(x, y)dx^{i} \in Z^{1}(SM)$ is the associate 1-form on $SM.$
 Using Theorem \ref{Pd}, three is a unique minimizer, say $f_{0}\in C^\infty(SM)$ such that $ ||\theta_{i}(x, y)dx^{i}-d_{H}f_{0}||^{2} $ is minimized. For all
$f\in C^\infty(SM)$ and $ t\in \mathbb{R}$ we have
\begin{align*}
&\frac{d}{dt}||\theta_{i}(x, y)dx^{i}-d_{H}f_{0}-td_{H}f||^{2}=\frac{d}{dt}(\theta_{i}(x, y)dx^{i}-d_{H}f_{0}-td_{H}f, \theta_{i}(x, y)dx^{i}-d_{H}f_{0}-td_{H}f)\\
&=\frac{d}{dt}[(\theta_{i}(x, y)dx^{i}-d_{H}f_{0}, \theta_{i}(x, y)dx^{i}-d_{H}f_{0})-2t(\theta_{i}(x, y)dx^{i}-d_{H}f_{0}, d_{H}f)\\
& +t^{2}(d_{H}f, d_{H}f)].
\end{align*} 																																									 Since $ ||\theta_{i}(x, y)dx^{i}-d_{H}f_{0}-td_{H}f||^{2} $ has a unique minimum at $ t=0 $, we deduce
\begin{equation}\label{eqz2}
(\theta_{i}(x, y)dx^{i}-d_{H}f_{0}, d_{H}f)=0,
\end{equation}
  for all $f\in C^\infty(SM).$
 On the other hand
\begin{equation}\label{eqz3}
(\theta_{i}(x, y)dx^{i}-d_{H}f_{0}, d_{H}f)=(\delta_{H}(\theta_{i}(x, y)dx^{i}-d_{H}f_{0}), f).
\end{equation}
The equations (\ref{eqz2}) and (\ref{eqz3}) yield $\delta_{H}(\theta_{i}(x, y)dx^{i}-d_{H}f_{0})=0$ and the proof is complete.
\\
%IS IT CORRECT FOR $Z_H^{1}(SM) $?\\
\hspace {\stretch{1}}$\Box$\\
%\end{proof}
%The dimension of the space of all harmonic forms of degree p is called the \emph{pth Betti number} of the manifold.
We then prove the corollary.\\
\emph{\textbf{Proof of Corollary\ref{Cor;BetiNum}.}}~
Let $ (M, F) $ be a closed orientable Finsler manifold and $X$ a harmonic vector field  related to $F$.
Assuming  $ \tilde{Ric} >0$, the second part of  Theorem \ref{theorem2} asserts  that the harmonic vector field $X$ related to $F$ vanishes identically. Theorem \ref{Th;harmonic} yields that the dimension of the space of all harmonic forms of degree one is the first Betti number of the manifold. Hence the first Betti number is $ b_{1}=0.$
\hspace {\stretch{1}}$\Box$\\
%\setcounter{theorem}{2}
%\setcounter{section}{1}
%\begin{corollary}
%Let $ (M, F) $ be a closed orientable Finslerian manifold with $ \tilde{Ric} >0$. The first Betti number vanishes.
%\end{corollary}
%\vspace{1cm}

\vspace{2cm}

  Behroz Bidabad\\
Professor of Department of Mathematics and Computer Sciences\\
Amirkabir University of Technology (Tehran Polytechnic),
424 Hafez Ave. 15914 Tehran, Iran.
E-mail:  bidabad@aut.ac.ir\\
Associate Researcher of (IMT) Institut de Mathematique de Toulouse,\\
Universit\'{e} Paul Sabatier, 118 route de Narbonne - F-31062 Toulouse, France.\\
E-mail: behroz.bidabad@math.univ-toulouse.fr\\
%\vspace{0.5cm}

Mir Ahmad Mirshafeazadeh,\\
Department of Mathematics, Shabestar Branch,
Islamic Azad University, Shabestar, Iran. E-mail: ah.mirshafeazadeh@gmail.com
\end{document}